\documentclass[12pt]{article}

\usepackage{amssymb}
\usepackage{amsmath} 
\usepackage{amsfonts}
\usepackage{amsthm}
\usepackage[colorlinks]{hyperref} 

\newtheorem{thm}{Theorem}[section]
\newtheorem{cor}[thm]{Corollary}
\newtheorem{lem}[thm]{Lemma}

\theoremstyle{definition}
\newtheorem{defn}[thm]{Definition}

\newtheorem{remk}[thm]{Remark}

\newcommand{\HH}{\mathcal{H}}

\newcommand{\R}{\mathbb{R}}

\newcommand{\leqs}{\leqslant}
\newcommand{\geqs}{\geqslant}

\newcommand{\ap}{\alpha}

\newcommand{\ep}{\epsilon}
\newcommand{\ld }{\lambda}

\newcommand{\sm}{\,\sigma\,}

\newcommand{\norm}[1]{\lVert #1 \rVert}

\DeclareMathOperator{\Sp}{Sp}

\begin{document}

\begin{center}
{\Large Positivity, Betweenness and Strictness }\\

\vspace{0.2cm}

{\Large of Operator Means}
\end{center}

\begin{center}
{\large Pattrawut Chansangiam}\footnote{email: kcpattra@kmitl.ac.th;
Department of Mathematics, Faculty of Science,
King Mongkut's Institute of Technology Ladkrabang,
    Bangkok 10520, Thailand.}
\end{center}

\noindent \textbf{Abstract:}
An operator mean is a binary operation assigned to each pair of positive operators
satisfying monotonicity, continuity from above, the transformer inequality
and the fixed-point property.
It is well known that there are one-to-one correspondences between
operator means, operator monotone functions and Borel measures.
In this paper, we provide various characterizations for the concepts of positivity, betweenness and strictness of operator means in terms of operator monotone functions, Borel measures and certain operator equations.
\\



\section{Introduction}

According to the definition of a mean for positive real numbers in \cite{Toader-Toader}, a mean $M$ is defined to be
satisfied the following properties
    \begin{itemize}
        \item   \emph{positivity}: $s,t>0 \implies M(s,t)>0$ ;
        \item   \emph{betweenness}: $s \leqs t \implies s \leqs M(s,t) \leqs t$.
    \end{itemize}
A mean $M$ is said to be
    \begin{itemize}
    	\item   \emph{strict at the left} if for each $a>0,b>0$, 
        $M(a,b) = a \implies a=b$
        \item   \emph{strict at the right} if for each 
        		$a>0,b>0$, $M(a,b) = b \implies a=b$,
        \item   \emph{strict} if it is both strict at the right and the left.
    \end{itemize}
A general theory of operator means 
was given by Kubo and Ando \cite{Kubo-Ando}.
Let $B(\HH)$ be the algebra of bounded linear operators on
a Hilbert space $\HH$. The set of positive operators on $\HH$ is denoted by $B(\HH)^+$.
Denote the spectrum of an operator $X$ by $\Sp(X)$.
For Hermitian operators $A,B \in B(\HH)$,
the partial order $A \leqs B$ indicates that $B-A \in B(\HH)^+$.
The notation $A>0$ suggests that $A$ is a strictly positive operator.
A \emph{connection} is a binary operation $\sm$ on $B(\HH)^+$
such that for all positive operators $A,B,C,D$:
\begin{enumerate}
	\item[(M1)] \emph{monotonicity}: $A \leqs C, B \leqs D \implies A \sm B \leqs C \sm D$
	\item[(M2)] \emph{transformer inequality}: $C(A \sm B)C \leqs (CAC) \sm (CBC)$
	\item[(M3)] \emph{continuity from above}:  for $A_n,B_n \in B(\HH)^+$,
                if $A_n \downarrow A$ and $B_n \downarrow B$,
                 then $A_n \sm B_n \downarrow A \sm B$.
                 Here, $A_n \downarrow A$ indicates that $A_n$ is a decreasing sequence
                 and $A_n$ converges strongly to $A$.
\end{enumerate}
Two trivial examples are the left-trivial mean $\omega_l : (A,B) \mapsto A$ and the right-trivial mean $\omega_r: (A,B) \mapsto B$.
Typical examples of a connection are the sum $(A,B) \mapsto A+B$ and the parallel sum
\begin{align*}
    A \,:\,B = (A^{-1}+B^{-1})^{-1}, \quad A,B>0.
\end{align*}
In fact, the parallel sum, introduced by Anderson and Duffin \cite{Anderson-Duffin} for analyzing electrical networks, is a model 
for general connections.
From the transformer inequality, every connection is \emph{congruence invariant} in the sense that for each $A,B \geqs 0$
and $C>0$ we have
\begin{align*}
	C(A \sm B)C \:=\: (CAC) \sm (CBC).
\end{align*}

A \emph{mean} in Kubo-Ando sense is a connection $\sigma$ with 
fixed-point property $A \sm A =A$ for all $A \geqs 0$.
The class of Kubo-Ando means cover many well-known means
in practice, e.g.
	\begin{itemize}
		\item	$\ap$-weighted arithmetic means: $A \triangledown_{\ap} B = (1-\ap)A + \ap B$
		\item	$\ap$-weighted geometric means:
    			$A \#_{\ap} B =  A^{1/2}
    			({A}^{-1/2} B  {A}^{-1/2})^{\ap} {A}^{1/2}$
		\item	$\ap$-weighted harmonic means: $A \,!_{\ap}\, 
				B = [(1-\ap)A^{-1} + \ap B^{-1}]^{-1}$ 
		\item	logarithmic mean: $(A,B) \mapsto A^{1/2}f(A^{-1/2}BA^{-1/2})A^{1/2}$ 
		where $f: \R^+ \to \R^+$, $f(x)=(x-1)/\log{x}$, $f(0) \equiv 0$ and $f(1) \equiv 1$.
		Here, $\R^+=[0, \infty)$.
	\end{itemize}

It is a fundamental that there are one-to-one correspondences between the following objects:
\begin{enumerate}
	\item[(1)]	operator connections on $B(\HH)^+$
	\item[(2)]	operator monotone functions from $\R^+$ to $\R^+$
	\item[(3)]	finite (positive) Borel measures on $[0,1]$.
\end{enumerate}
Recall that a continuous function $f: \R^+ \to \R^+$ is said to be \emph{operator monotone} if
\begin{align*}
	A \leqs B \implies f(A) \leqs f(B)
\end{align*}
for all positive operators $A,B \in B(\HH)$ and for all Hilbert spaces $\HH$.
This concept was introduced in \cite{Lowner}; see also \cite{Bhatia,Hiai,Hiai-Yanagi}.
Every operator monotone function from $\R^+$ to $\R^+$
is always differentiable (see e.g. \cite{Hiai}) and concave in usual sense (see \cite{Hansen-Pedersen}).

A connection $\sigma$ on $B(\HH)^+$ can be characterized
via operator monotone functions as follows:

\begin{thm}[\cite{Kubo-Ando}] \label{thm: Kubo-Ando f and sigma}
	Given a connection $\sigma$, there is a unique operator monotone function $f: \R^+ \to \R^+$ satisfying
				\begin{align*}
    			f(x)I = I \sm (xI), \quad x \geqs 0.
				\end{align*}
	Moreover, the map $\sigma \mapsto f$ is a bijection.
\end{thm}

We call $f$ the \emph{representing function} of $\sigma$.
A connection also has a canonical characterization with respect to a Borel measure 
via a meaningful integral representation as follows.

\begin{thm}[\cite{Pattrawut}] \label{thm: 1-1 conn and measure}
Given a finite Borel measure $\mu$ on $[0,1]$, the binary operation
	\begin{align}
		A \sm B = \int_{[0,1]} A \,!_t\, B \,d \mu(t), \quad A,B \geqs 0
		\label{eq: int rep connection}
	\end{align}
	is a connection on $B(\HH)^+$.
	Moreover, the map $\mu \mapsto \sigma$ is bijective,
	in which case the representing function of $\sigma$ is given by
	\begin{align}
		f(x) = \int_{[0,1]} (1 \,!_t\, x) \,d \mu(t), \quad x \geqs 0. \label{int rep of OMF}
	\end{align}
\end{thm}
We call $\mu$ the \emph{associated measure} of $\sigma$.

\begin{thm} \label{thm: char of mean}
	Let $\sigma$ be a connection
on $B(\HH)^+$ with representing function $f$ and representing measure $\mu$. Then the following statements are equivalent.
\begin{enumerate}
    \item[(1)]    $I \sm I=I$ ;
    \item[(2)]    $A \sm A =A$ for all $A \in B(\HH)^+$ ;
    \item[(3)]    $f$ is normalized, i.e., $f(1)=1$ ;
    \item[(4)]    $\mu$ is normalized, i.e., $\mu$ is a probability measure.
\end{enumerate}
\end{thm}
Hence every mean can be regarded as an average of weighted harmonic means.
From \eqref{eq: int rep connection} and \eqref{int rep of OMF} in Theorem \ref{thm: 1-1 conn and measure}, $\sigma$
and $f$ are related by
\begin{align}
	f(A) \:=\: I \sm A, \quad A \geqs 0.  \label{eq: f(A) = I sm A}
\end{align}

In this paper, we provide various characterizations of the concepts of positivity, betweenness and strictness of operator means in terms of operator monotone functions, Borel measures and certain operator equations.
It turns out that every mean satisfies the positivity property.
The betweenness is a necessary and sufficient condition for a connection
to be a mean. A mean is strict at the left (right) if and only if
it is not the left-trivial mean (the right-trivial mean, respectively).

\section{Positivity}

We say that a connection $\sigma$ satisfies the \emph{positivity property} if
\begin{align*}
    A,B>0 \implies A \sm B >0.
\end{align*}
Recall that the \emph{transpose} of a connection $\sigma$ is the connection
\begin{align*}
	(A,B) \mapsto B \sm A.
\end{align*}
If $f$ is the representing function of $\sigma$, then the representing function of its transpose is given by 
	\begin{align*}
		g(x) = xf(\frac{1}{x}), \quad x>0
	\end{align*}
	and $g(0)$ is defined by continuity. 

\begin{thm} \label{remk: A,B>0 imply A sm B>0}
Let $\sigma$ be a connection
on $B(\HH)^+$ with representing function $f$ and associated measure $\mu$.
Then the following statements are equivalent:
\begin{enumerate}
    \item[(1)]    $\sigma$  satisfies the positivity property ;
    \item[(2)]    $I \sm I >0$ ;
    \item[(3)]    $\sigma \neq 0$ (here, $0$ is the zero connection $(A,B) \mapsto 0$) ;
    \item[(4)]    for all $A \geqs 0$, $A \sm A =0 \implies A=0$ (\emph{positive definiteness}) ;
    \item[(5)]    for all $A \geqs 0$, $A \sm I =0 \implies A=0$ ;
    \item[(6)]    for all $A \geqs 0$, $I \sm A =0 \implies A=0$ ;
    \item[(7)]    for all $A \geqs 0$ and $B>0$, $A \sm B =0 \implies A=0$ ;
    \item[(8)]    for all $A > 0$ and $B \geqs 0$, $A \sm B =0 \implies B=0$ ;
    \item[(9)]    $f \neq 0$ (here, $0$ is the function $x \mapsto 0$) ;
    \item[(10)]    $x>0 \implies f(x)>0$ ;
    \item[(11)]    $\mu([0,1])>0$. 
\end{enumerate}
\end{thm}
\begin{proof}
    The implications (1) $\Rightarrow$ (2) $\Rightarrow$ (3), (4) $\Rightarrow$ (3),
    (7) $\Rightarrow$ (5) $\Rightarrow$ (3),
    (8) $\Rightarrow$ (6) $\Rightarrow$ (3)  and
    (10) $\Rightarrow$ (9) are clear.
    Using the integral representations in Theorem \ref{thm: 1-1 conn and measure}, it is straightforward to verify that
	 the representing function of the zero connection
	 $
	 	 0 \;:\; (A,B) \mapsto 0
	 $
	 is the constant function $f \equiv 0$
	 and its associated measure is the zero measure.
    Hence, we have the equivalences (3) $\Leftrightarrow$ (9) $\Leftrightarrow$ (11).

    (9) $\Rightarrow$ (10): Assume $f \neq 0$. 
    Suppose that there is an $a>0$ such that $f(a)=0$.
    Then $f(x)=0$ for all $x \leqs a$.
   The concavity of $f$ implies that $f(x) \leqs 0$ for all $x \geqs a$.
   The monotonicity of $f$ forces that $f(x)=0$ for all $x \geqs a$.
   Hence $f = 0$, a contradiction.

    (5) $\Rightarrow$ (7): Assume (5). Let $A \geqs 0$ and $B >0$ be such that $A \sm B =0$. Then
    \begin{align*}
        0 = B^{1/2}(B^{-1/2} A B^{-1/2} \sm I) B^{1/2}
    \end{align*}
    and $B^{-1/2} A B^{-1/2} \sm I =0$. Now, (5) yields $B^{-1/2} A B^{-1/2} = 0$, i.e. $A=0$.

    (6) $\Rightarrow$ (8): It is similar to (5) $\Rightarrow$ (7).

    (10) $\Rightarrow$ (1): Assume that $f(x)>0$ for all $x>0$. 
    Since $\Sp(f(A))=f(\Sp(A))$ by spectral mapping theorem, we have
    $f(A)>0$  for all $A>0$. 
       Hence, for each $A,B>0$,
    \begin{align*}
        A \sm B = A^{1/2} f(A^{-1/2}BA^{-1/2})A^{1/2} >0.
    \end{align*}
    
    (10) $\Rightarrow$ (4): Assume (10). Let $A \geqs 0$ be such that $A \sm A=0$. Then
    \begin{align*}
        A \sm A &= \lim_{\ep \downarrow 0} A_{\ep} \sm A_{\ep}
                = \lim_{\ep \downarrow 0} A_{\ep}^{1/2} (I \sm I) A_{\ep}^{1/2} \\
                &= \lim_{\ep \downarrow 0} f(1) A_{\ep}
                = f(1) A,
    \end{align*}
    here $A_{\ep} \equiv A + \ep I$.
    Since $f(1)>0$, we have $A=0$.
    
	(10) $\Rightarrow$ (5): Assume (10). Let $A \geqs 0$ be such that 
	$A \sm I =0$.
	Then $g(A)=0$ where $g$ is the representing function of the transpose of $\sigma$.
	We see that $g(x)>0$ for $x>0$. The injectivity of functional calculus implies that $g(\ld)=0$
	for all $\ld \in \Sp(A)$.
	We conclude $\Sp(A)=0$, i.e. $A=0$.
	
	(10) $\Rightarrow$ (6): Assume (10). Let $A \geqs 0$ be such that $I \sigma A =0$.
    Then $f(A)=0$. By the injectivity of functional calculus, we have $f(\ld)=0$ for all $\ld \in \Sp(A)$.
    The assumption (10) implies that $\Sp(A)=\{0\}$. Thus, $A=0$.
\end{proof}

\begin{remk}
    It is not true that $\sigma \neq 0$ implies the condition that for all $A \geqs 0$ and $B \geqs 0$,
    $A \sm B =0 \implies A=0$. Indeed, take $\sm$ to be the geometric mean and
    \begin{equation*}
        A= \left(\begin{array}{cc} 1 & 0 \\0 & 0 \\\end{array}\right), \quad
        B= \left(\begin{array}{cc} 0 & 0 \\0 & 1 \\\end{array}\right).
    \end{equation*}
\end{remk}

\section{Betweenness}

We say that a connection $\sigma$ satisfies the \emph{betweenness property}
if for each $A,B \geqs 0$,
\begin{align*}
	A \leqs B \implies A \leqs A \sm B \leqs B.
\end{align*}
 
By Theorem \ref{remk: A,B>0 imply A sm B>0}, every mean enjoys the positivity property.
In fact, the betweenness property is a necessary and sufficient condition for a connection to be a mean:

\begin{thm} The following statements are equivalent for a connection $\sigma$ with representing function $f$:
\begin{enumerate}
    \item[(1)]    $\sigma$ is a mean ;
    \item[(2)]    $\sigma$ satisfies the betweenness property ;
    \item[(3)]   for all $A \geqs 0$, $A \leqs I \implies A \leqs A \sm I \leqs I$ ;
    \item[(4)]   for all $A \geqs 0$, $I \leqs A \implies I \leqs I \sm A \leqs A$ ;
    \item[(5)]   for all $t\geqs 0$, $1 \leqs t \implies 1 \leqs f(t) \leqs t$ ;
    \item[(6)]   for all $t\geqs 0$, $t \leqs 1 \implies t \leqs f(t) \leqs 1$ ;
    \item[(7)]	 for all $A,B \geqs 0$,  $A \leqs B \implies \norm{A} \leqs \norm{A \sm B} \leqs \norm{B}$ ;
    \item[(8)]	 for all $A \geqs 0$,  
    			 $A \leqs I \implies \norm{A} \leqs \norm{A \sm I} \leqs 1$ ;
    \item[(9)]	 for all $A \geqs 0$,  
    			 $I \leqs A \implies 1 \leqs \norm{I \sm A} \leqs \norm{A}$ ; 
    \item[(10)]	 the only solution $X>0$ to the equation $X \sm X = I$ is $X=I$ ;
    \item[(11)]  for all $A>0$, the only solution $X>0$ to the equation $X \sm X =A$ is $X=A$.
\end{enumerate}
\end{thm}
\begin{proof}
The implications (2) $\Rightarrow$ (3), (2) $\Rightarrow$ (4), (2) $\Rightarrow$ (7) $\Rightarrow$ (8) 
and (11) $\Rightarrow$ (10) $\Rightarrow$ (1) 
are clear.

    (1) $\Rightarrow$ (2): Let $A,B \geqs 0$ be such that $A \leqs B$.
    The fixed point property and the monotonicity of $\sigma$ yield
    \begin{align*}
    	A = A \sm A \leqs A \sm B \leqs B \sm B =B.
    \end{align*}

    (3) $\Rightarrow$ (1): Since $I \leqs I$, we have $I \leqs I \sm I \leqs I$, i.e. $I \sm I =I$. Hence $\sigma$ is a mean by Theorem \ref{thm: char of mean}.
    
 	(4) $\Rightarrow$ (1): It is similar to  (3) $\Rightarrow$ (1).
 	
 	(8) $\Rightarrow$ (1): We have $1=\norm{I} \leqs \norm{I \sm I} \leqs 1$.
 	Hence, 
 	\begin{align*}
 		f(1) = \norm{f(1) I} = \norm{I \sm I} = 1.
 	\end{align*}
 	Therefore, $\sigma$ is a mean by Theorem \ref{thm: char of mean}.
 	
 	(2) $\Rightarrow$ (5): If $t \geqs 1$, then $I \leqs I \sm (tI) \leqs tI$
    which is $I \leqs f(t)I \leqs tI$, i.e. $1 \leqs f(t) \leqs t$.

    (5) $\Rightarrow$ (1): We have $f(1)=1$.

    (2) $\Rightarrow$ (6) $\Rightarrow$ (1): It is similar to (2) $\Rightarrow$ (5) $\Rightarrow$ (1).
 	
    (1) $\Rightarrow$ (11): Let $A>0$. Consider $X>0$ such that $X \sm X =A$.
    Then by the congruence invariance of $\sigma$, we have
    \begin{align*}
    	X = X^{1/2} (I \sm I) X^{1/2} = X \sm X = A.
	\end{align*}     
\end{proof}

\begin{remk}
    For a connection $\sigma$ and $A,B \geqs 0$, the operators $A,B$ and $A \sm B$ need not be comparable.
    The previous theorem tells us that if $\sigma$ is a mean, then the condition $0 \leqs A \leqs B$ guarantees
    the comparability between $A,B$ and $A \sigma B$.
\end{remk}

\section{Strictness}

We consider the strictness of Kubo-Ando means as that for scalar means in \cite{Toader-Toader}:

\begin{defn}
    A mean $\sigma$ on $B(\HH)^+$ is said to be
    \begin{itemize}
    	\item   \emph{strict at the left} if for each $A>0,B>0$, $A \sm B = A \implies A=B$,
        \item   \emph{strict at the right} if for each $A>0,B>0$, 
        $A \sm B = B \implies A=B$,

        \item   \emph{strict} if it is both strict at the right and the left.
    \end{itemize}
\end{defn}

The following lemmas are easy consequences of the fact that
an operator monotone function
$f: \R^+ \to \R^+$ is always monotone, concave and differentiable.

\begin{lem} \label{lem 1}
	If $f:\R^+ \to \R^+$ is an operator monotone function such that $f$ is a constant 
	on an interval $[a,b]$ with $a<b$, then $f$ is a constant on $\R^+$.
\end{lem}


\begin{lem} \label{lem 2}
	If $f:\R^+ \to \R^+$ is an operator monotone function such that $f$ is the identity function 
	on an interval $[a,b]$ with $a<b$, then $f$ is the identity on $\R^+$.
\end{lem}

\begin{thm} \label{thm: strict_left trivial}
Let $\sigma$ be a mean with representing function $f$ and associated measure $\mu$.
Then the following statements are equivalent: 
\begin{enumerate}
    \item[(1)]   $\sigma$ is strict at the left ;
    \item[(2)]   $\sigma$ is not the left-trivial mean ;
    \item[(3)]   for all $A \geqs 0$, $I \sm A =I \implies A=I$ ;
    \item[(4)]   for all $A > 0$, $A \sm I =A \implies A=I$ ;
    \item[(5)]   for all $A>0$ and $B \geqs 0$, $A \sm B =A \implies A=B$ ;
    \item[(6)]   for all $A \geqs 0$, $I \leqs I \sm A \implies I \leqs A$ ;
    \item[(7)]   for all $A \geqs 0$, $I \sm A \leqs I \implies A \leqs I$ ;
    \item[(8)]   for all $A > 0$, $A \leqs A \sm I \implies A \leqs I$ ;
    \item[(9)]   for all $A > 0$, $A \sm I \leqs A \implies I \leqs A$ ;
    \item[(10)]   for all $A>0$ and $B \geqs 0$, $A \leqs  A \sm B \implies A \leqs B$ ;
    \item[(11)]   for all $A>0$ and $B \geqs 0$, $A \sm B \leqs  A \implies B \leqs A$ ;
    \item[(12)]	  $f$ is not the constant function $x \mapsto 1$ ;
    \item[(13)]   for all $x \geqs 0$, $f(x)=1 \implies x=1$ ;
    \item[(14)]   for all $x \geqs 0$, $f(x) \geqs 1 \implies x \geqs 1$ ;
    \item[(15)]   for all $x \geqs 0$, $f(x) \leqs 1 \implies x \leqs 1$ ;
    \item[(16)]	  $\mu$ is not the Dirac measure at $0$.
\end{enumerate}
\end{thm}
\begin{proof}
	It is clear that $(5) \Rightarrow (1)$ and each of $(1),(4),(6)-(11)$  implies $(2)$. Also, each of $(13)-(15)$ implies $(12)$.
	
	(2) $\Rightarrow$ (3): Let $A \geqs 0$ be such that $I \sm A =I$. Then $f(A)=I$.
	Hence, $f(\ld)=1$ for all $\ld \in \Sp(A)$.
	Suppose that $\ap \equiv \inf \Sp(A) < r(A)$ where $r(A)$ is the spectral radius of $A$.
	Then $f(x)=1$ for all $x \in [\ap, r(A)]$.
	It follows that $f \equiv 1$ on $\R^+$ by Lemma \ref{lem 1}.
	This contradicts the assumption (2).
	We conclude $\ap =r(A)$, i.e. $\Sp(A)=\{\ld \}$ for some $\ld \geqs 0$.
	Suppose now that $\ld <1$. Since $f(1)=1$, we have that $f$ is a constant on the interval $[\ld,1]$.
	Again, Lemma \ref{lem 1} implies that $f \equiv 1$ on $\R^+$, a contradiction.
	Similarly, $\ld >1$ gives a contradiction. 
	Thus $\ld =1$, which implies $A=I$. 
	
	(2) $\Rightarrow$ (4): Let $A > 0$ be such that $A \sm I =A$. Then $g(A)=I$ where $g$ is the representing
	function of the transpose of $\sigma$.
	Hence, $g(\ld)=\ld$ for all $\ld \in \Sp(A)$.
	Suppose that $\ap \equiv \inf \Sp(A) < r(A)$.
	Then $g(x)=x$ for all $x \in [\ap, r(A)]$.
	It follows that $g(x)=x$ on $\R^+$ by Lemma \ref{lem 2}.
	Hence, the transpose of $\sigma$ is the right-trivial mean. 
	This contradicts the assumption (2).
	We conclude $\ap =r(A)$, i.e. $\Sp(A)=\{\ld \}$ for some $\ld \geqs 0$.
	Suppose now that $\ld <1$. Since $g(1)=1$, we have that $g(x)=x$ on the interval $[\ld,1]$.
	Lemma \ref{lem 2} forces $g(x)=x$ on $\R^+$, a contradiction.
	Similarly, $\ld >1$ gives a contradiction. 
	Thus $\ld =1$, which implies $A=I$.
	
	(3) $\Rightarrow$ (5): Use the congruence invariance of $\sigma$.
	
	(2) $\Rightarrow$ (6): Assume that $\sigma$ is not the left-trivial mean.
	Let $A \geqs 0$ be such that $I \sm A \leqs I$. Then $f(A) \geqs I$.
	The spectral mapping theorem implies that $f(\ld) \geqs 1$ for all $\ld \in \Sp(A)$.
	Suppose there exists a $ t \in \Sp(A)$ such that $t<1$. Since $f(t) \leqs f(1)=1$, we have
	$f(t)=1$. It follows that $f(x)=1$ for $t \leqs x \leqs 1$.
	By Lemma \ref{lem 1}, $f \equiv 1$ on $\R^+$, a contradiction.
	We conclude $\ld \geqs 1$ for all $\ld \in \Sp(A)$, i.e. $A \geqs I$.
	
	(2) $\Rightarrow$ (7): It is similar to (2) $\Rightarrow$ (6).
	
	(6) $\Rightarrow$ (8): Assume (6). Let $A>0$ be such that 
	$A  \leqs A \sm I$.
	Then
	\begin{align*}
		A \leqs A^{1/2} (I \sm A^{-1})A^{1/2},
	\end{align*}
	which implies $I \leqs I \sm A^{-1}$. By (6), we have $I \leqs A^{-1}$ or $A \leqs I$.
	
	(7) $\Rightarrow$ (9): It is similar to (6) $\Rightarrow$ (8).
	
	(6) $\Rightarrow$ (10): Use the congruence invariance of $\sigma$.
	
	(7) $\Rightarrow$ (11): Use the congruence invariance of $\sigma$.
	
	(2) $\Leftrightarrow$ (12) $\Leftrightarrow$ (16): 
	Note that the representing function of the left-trivial mean
	is the constant function $f \equiv 1$. Its associated measure is the Dirac measure at $0$.
	
	(2) $\Rightarrow$ (13). Assume (2). Let $t \geqs 0$ be such that $f(t)=1$.
	Suppose that there is a $\ld \neq 1$ such that $f(\ld)=1$.
	It follows that $f(x) =1$ for all $x$ lying between $\ld$ and $1$.
	Lemma \ref{lem 1} implies that $f \equiv 1$ on $\R^+$, contradicting the assumption (2). 
	
	(2) $\Rightarrow$ (14), (15). Use the argument in the proof (2) $\Rightarrow$ (13). 
\end{proof}

\begin{thm} \label{thm: strict_right trivial}
Let $\sigma$ be a mean with representing function $f$ and associated measure $\mu$.
Then the following statements are equivalent: 
\begin{enumerate}
    \item[(1)]   $\sigma$ is strict at the right ;
    \item[(2)]   $\sigma$ is not the right-trivial mean ;
    \item[(3)]   for all $A \geqs 0$, $A \sm I =I \implies A=I$ ;
    \item[(4)]   for all $A > 0$, $I \sm A = A \implies A=I$ ;
    \item[(5)]   for all $A \geqs 0$ and $B > 0$, $A \sm B = B \implies A=B$ ;
    \item[(6)]   for all $A \geqs 0$, $I \leqs A \sm I \implies I \leqs A$ ;
    \item[(7)]   for all $A \geqs 0$, $A \sm I \leqs I \implies A \leqs I$ ;
    \item[(8)]   for all $A > 0$, $A \leqs I \sm A \implies A \leqs I$ ;
    \item[(9)]   for all $A > 0$, $I \sm A \leqs A \implies I \leqs A$ ;
    \item[(10)]   for all $A \geqs 0$ and $B > 0$, $B \leqs  A \sm B \implies B \leqs A$ ;
    \item[(11)]   for all $A \geqs 0$ and $B > 0$, $A \sm B \leqs  B \implies A \leqs B$ ;
    \item[(12)]	  $f$ is not the identity function $x \mapsto x$ ;
    \item[(13)]	  $\mu$ is not the associated measure at $1$.
\end{enumerate}
\end{thm}
\begin{proof}
    Replace $\sigma$ by its transpose in the previous theorem.
\end{proof}

We immediately get the following corollaries.

\begin{cor}
    A mean is strict if and only if it is non-trivial.
\end{cor}

\begin{cor} \label{thm: strict_cor 2}
    Let $\sigma$ be a non-trivial mean. For each $A>0$ and $B>0$, the following statements are equivalent:
    \begin{enumerate}
        \item[(i)]   $A = B$,
        \item[(ii)]   $A \sm B =A$,
        \item[(iii)]   $A \sm B = B$,
        \item[(iv)]   $B \sm A = A$,
        \item[(v)]   $B \sm A = B$.
    \end{enumerate}
\end{cor}

The next result is a generalization of \cite[Theorem 4.2]{Fiedler}, in which
the mean $\sigma$ is the geometric mean.

\begin{cor} \label{thm: strict_cor 3}
    Let $\sigma$ be a non-trivial mean. For each $A>0$ and $B>0$, the following statements are equivalent:
    \begin{enumerate}
        \item[(i)]   $A \leqs B$,
        \item[(ii)]   $A \leqs A \sm B$,
        \item[(iii)]   $A \sm B \leqs B$,
        \item[(iv)]   $A \leqs B \sm A $,
        \item[(v)]   $B \sm A \leqs B$.
    \end{enumerate}
\end{cor}

\begin{remk}
    (i) It is not true that if $\sigma$ is not the left-trivial mean then for all $A \geqs 0$ and $B \geqs 0$,
    $A \sm B =A \implies A=B$. Indeed, take $\sigma$ to be the geometric mean, $A=0$ and
    \begin{equation*}
        B= \left(\begin{array}{cc} 0 & 0 \\0 & 1 \\\end{array}\right).
    \end{equation*}
    The case of right-trivial mean is just the same.

    (ii) The assumption of invertibility of $A$ or $B$
    in Corollary \ref{thm: strict_cor 2} cannot be omitted, as a counter example in (i) shows.
    As well, the invertibility of $A$ or $B$ in Corollary 
    \ref{thm: strict_cor 3} cannot be omitted.
    Consider the geometric mean and
    \begin{equation*}
        A= \left(\begin{array}{cc} 1 & 0 \\0 & 0 \\\end{array}\right),  \quad
        B= \left(\begin{array}{cc} 0 & 0 \\0 & 1 \\\end{array}\right).
    \end{equation*}
\end{remk}


\begin{thebibliography}{40}

\bibitem{Anderson-Duffin}   W.N. Anderson and R.J. Duffin, Series and parallel addition of matrices,
J. Math. Anal. Appl. 26(1969), 576--594.








\bibitem{Bhatia} R. Bhatia, Matrix analysis, Springer-Verlag, New York, 1997.


\bibitem{Pattrawut} P.~Chansangiam and W.~Lewkeeratiyutkul, 
Operator connections and Borel measures on the unit interval,
arXiv:submit/0588472 [math.FA] 7 Nov 2012.


\bibitem{Fiedler} M. Fielder and V. Ptak, 
A new positive definite geometric mean of two positive definite matrices,
Linear Algebra Appl. 251(1997), 1--20.




 





\bibitem{Hansen-Pedersen} F.~Hansen and G.K.~Pedersen, Jenden's inequality for operators and L\"{o}wner's theorem, Math. Ann. 258(1982),  229--241.

\bibitem{Hiai}  F. Hiai, Matrix analysis: matrix monotone functions, matrix means, and
majorizations, Interdiscip. Inform. Sci. 16(2010), 139--248.


\bibitem{Hiai-Yanagi}   F. Hiai and K.Yanagi, Hilbert spaces and linear operators, Makino Pub.Ltd., 1995.



\bibitem{Kubo-Ando} F. Kubo and T. Ando, Means of positive linear operators, Math. Ann. 246(1980),
205--224.


\bibitem{Lowner}  C.  L\"{o}wner,  \"{U}ber monotone matrix funktionen. Math. Z. 38(1934), 177--216.







\bibitem{Toader-Toader} G. Toader and S. Toader,  Greek means and the arithmetic-geometric mean, RGMIA Monographs, Victoria University, 2005. 

\end{thebibliography}
\end{document}